\documentclass[14pt]{amsart}
\usepackage{amssymb,amsmath,epsfig,graphics,mathrsfs}

\usepackage{amsthm,epsfig,esint}
\usepackage{ulem}

\usepackage{fancyhdr}
\usepackage{hyperref}
\hypersetup{
 colorlinks   = true,
 urlcolor     = blue,
 linkcolor    = blue,
 citecolor   = red ,
 bookmarksopen=true
}

\usepackage{amsmath}
\usepackage{amsfonts}
\usepackage{amssymb}
\usepackage{amsthm}
\usepackage{epsfig,graphics,mathrsfs}
\usepackage{graphicx}

\usepackage[usenames, dvipsnames]{color} 

\usepackage{hyperref}

\usepackage{ulem}

\theoremstyle{plain}
\newtheorem{thrm}{Theorem}[section]
\newtheorem*{thrm*}{Theorem}
\newtheorem{lemma}[thrm]{Lemma}
\newtheorem{prop}[thrm]{Proposition}
\newtheorem{cor}[thrm]{Corollary}
\theoremstyle{definition}

\theoremstyle{remark}
\newtheorem{rmrk}[thrm]{Remark}

\numberwithin{equation}{section}
\usepackage{color}

\begin{document}

\newcommand{\ddelta}{\delta}

\newcommand{\tx}{\tilde x}
\newcommand{\R}{\mathbb R}
\newcommand{\N}{\mathbb N}
\newcommand{\C}{\mathbb C}
\newcommand{\lie}{\mathcal G}
\newcommand{\hN}{\mathcal N}
\newcommand{\D}{\mathcal D}
\newcommand{\A}{\mathcal A}
\newcommand{\B}{\mathcal B}
\newcommand{\sL}{\mathcal L}
\newcommand{\sLi}{\mathcal L_{\infty}}

\newcommand{\G}{\Gamma}
\newcommand{\x}{\xi}

\newcommand{\eps}{\epsilon}
\newcommand{\al}{\alpha}
\newcommand{\be}{\beta}
\newcommand{\p}{\partial}  
\newcommand{\lig}{\mathfrak}

\def\dist{\mathop{\varrho}\nolimits}

\newcommand{\BCH}{\operatorname{BCH}\nolimits}
\newcommand{\Lip}{\operatorname{Lip}\nolimits}
\newcommand{\Hol}{C}                             
\newcommand{\lip}{\operatorname{lip}\nolimits}
\newcommand{\capQ}{\operatorname{Cap}\nolimits_Q}
\newcommand{\pCap}{\operatorname{Cap}\nolimits_p}
\newcommand{\Om}{\Omega}
\newcommand{\om}{\omega}
\newcommand{\half}{\frac{1}{2}}
\newcommand{\e}{\varepsilon}
\newcommand{\vn}{\vec{n}}
\newcommand{\X}{\Xi}
\newcommand{\tLip}{\tilde  Lip}

\newcommand{\Span}{\operatorname{span}}

\newcommand{\ad}{\operatorname{ad}}
\newcommand{\Hm}{\mathbb H^m}
\newcommand{\Hn}{\mathbb H^n}
\newcommand{\Hone}{\mathbb H^1}
\newcommand{\Lie}{\mathfrak}
\newcommand{\Layer}{V}
\newcommand{\hgrad}{\nabla_{\!H}}
\newcommand{\im}{\textbf{i}}
\newcommand{\nz}{\nabla_0}
\newcommand{\s}{\sigma}
\newcommand{\se}{\sigma_\e}

\newcommand{\ued}{u^{\e,\ddelta}}
\newcommand{\ueds}{u^{\e,\ddelta,\sigma}}
\newcommand{\tnabla}{\tilde{\nabla}}

\newcommand{\bx}{\bar x}
\newcommand{\by}{\bar y}
\newcommand{\bt}{\bar t}
\newcommand{\bs}{\bar s}
\newcommand{\bz}{\bar z}
\newcommand{\btau}{\bar \tau}

\newcommand{\LC}{\mbox{\boldmath $\nabla$}}
\newcommand{\Ne}{\mbox{\boldmath $n $}}
\newcommand{\nuo}{\mbox{\boldmath $n^0$}}
\newcommand{\nuu}{\mbox{\boldmath $n^1$}}
\newcommand{\nue}{\mbox{\boldmath $n $}}
\newcommand{\nuek}{\mbox{\boldmath $n^{\e_k}$}}
\newcommand{\dse}{\nabla^{H\Su, \e}}
\newcommand{\dso}{\nabla^{H\Su, 0}}
\newcommand{\tX}{\tilde X}

\newcommand\red{\textcolor{red}}
\newcommand\green{\textcolor{green}}

\newcommand{\Xie}{X psilon_i}
\newcommand{\Xje}{X psilon_j}
\newcommand{\Su}{\mathcal S}
\newcommand{\F}{\mathcal F}

\def\pa{\partial}
\def\Id{{\rm Id}\,}
\def\loc{{\rm loc}}
\def\div{{\rm div}\,}
\def\reg{{\rm reg}\,}
\def\regn{{\rm regn}\,}
\def\Cap{{\rm Cap}}

\title [Quasilinear equations in the Heisenberg group]{ 
Regularity theory of quasilinear elliptic and parabolic equations in the Heisenberg group}

\author[L. Capogna]{L. Capogna}
\address{Luca Capogna\\Department of Mathematics and Statistics, Smith College, Northampton, MA 01060, USA
}\email{lcapogna@smith.edu}

\author[G.  Citti]{G.  Citti}\address{Giovanna Citti\\Dipartimento di Matematica, Piazza Porta S. Donato 5, Universit\`a di Bologna, 
40126 Bologna, Italy;  Centro Linceo interdisciplinare Beniamino Segre, Accademia dei Lincei, Roma, Italy; CAMS, EHESS, Paris, France; GNAMPA INDAM} \email{giovanna.citti@unibo.it}

\author[X. Zhong]{X. Zhong}\address{Xiao Zhong\\Department of Mathematics and Statistics, University of Helsinki, 00014, University of Helsinki, Finland}
\email{xiao.x.zhong@helsinki.fi}

\thanks{LC was partially supported by NSF award  DMS1955992}
\thanks{GC was partially funded by  
Horizon 2020 Project ref. 777822: GHAIA}
\thanks{XZ was supported by the Academy of Finland, Centre of Excellence in Randomness and Structures.}

\begin{abstract}
This note provides a succinct survey of the existing literature concerning the H\"older regularity for the gradient of weak solutions of PDEs of the form 
$$\sum_{i=1}^{2n} X_i A_i(\nabla_0 u)=0 \text{ and } \partial_t u= \sum_{i=1}^{2n} X_i A_i(\nabla_0 u)$$ modeled on the   $p$-Laplacian in a domain  $\Omega$  in the Heisenberg group $\Hn$, with   $1\le p <\infty$, and of its parabolic counterpart. We present some open problems and outline some of the difficulties they present.
 \end{abstract}

\noindent{\it Review article:}

\maketitle

\tableofcontents

\section{Introduction and review of the literature}

If $\Omega\subset \R^n$, is an open set and $u:\Om\to \mathbb R$ is a sufficiently smooth function, the critical points of the  nonlinear  Dirichlet  energy ($p-$energy)
$$E_p(u,\Om)=\int_{\Omega} |\nabla u|^p ,$$
for $1<p<\infty$ are called $p-$harmonic functions, and are weak solutions of the Euler-Lagrange equation
\begin{equation}\label{euclidean}\sum_{i=1}^n \p_i (|\nabla u|^{p-2} \p_i u)=0\end{equation}
in $\Omega$. The $p-$energy plays a pivotal role in the study of nonlinear potential theory, through the notion of capacity, and as such it appears in many problems in analysis, geometry and in applied mathematics. 
The regularity of weak solutions of \eqref{euclidean} has a long history.  Since this is a quasilinear, degenerate elliptic PDE, the linear theory developed by De Giorgi \cite{De}, Nash \cite{Nash} and Moser \cite{moser} does not directly apply. An Harnack inequality for positive weak solutions of \eqref{euclidean}, and the ensuing H\"older regularity was established by Serrin in 1964 \cite{Se}, through ideas introduced earlier by Moser. The H\"older regularity of the gradient is more involved. The range $2\le p$ was studied by Uraltseva \cite{Ur} and by Uhlenbeck \cite{U}. The general $1<p<\infty$ case is due to the work of Di Benedetto \cite{Di}, Lewis \cite{Le} and Tolksdorff \cite{To}.  The (degenerate) parabolc counterpart
has been studied by many authors, and we refer to the book by Di Benedetto \cite{DB} and the more recent paper \cite{DBGV} for an extensive list of references. 

Moving beyond the Euclidean setting, the $p-$energy and $p-$harmonic functions continue to play a crucial role in the study of nonlinear potential theory, geometric function theory and geometry. For instance, in   Mostow rigidity theorem \cite{Mostow}, and its extensions by Pansu \cite{Pansu}, the $p-$energy and $p-$capacity of rings in the subriemannian Heisenberg group $\mathbb H^n$ appear naturally in the study of rigidity of homeomorphisms of the complex hyperbolic  space $H\mathbb C^{n+1}$. 

This survey is meant to provide a succinct introduction to the  the regularity theory for $p-$harmonic functions in the Heisenberg group, and for their parabolic counterpart. We aim at presenting the material roughly in  chronological order and to emphasize both the challenges and some of the ideas that have provided breakthroughs. 

\bigskip

The Heisenberg group $\mathbb H^n$ is a nilpotent Lie group. Its Lie algebra  $\mathfrak h$ can be identified with $\R^{2n+1}$, and admits a stratification $\mathfrak h= H\oplus V$, where $H$ and $V$ are respectively the  ($2n-$dimensional) horizontal and ($1-$dimensional) vertical tangent bundles, with $[H,H]=V,$  and $V$ is  the center of the algebra.  This space has a canonical left invariant measure $dx$ the Haar measure, which for nilpotent groups is the Lebsgue measure in $\mathfrak h$. Absolutely continuous curves $\gamma:[0,1]\to \mathbb H^n$ are called {\it horizontal} if their tangent is in $H$ a.e. with respect to Lebesgue measure in $\R$

 A subriemannian metric in $\mathbb H^n$ is a left-invariant Riemannian metric $g_0$ defined in $H$. The corresponding Carnot-Carath\'eodory control distance $d_0(x,y)$ is the shortest time it takes to travel from $x$ to $y$ along unit speed horizontal curves joining the two points. The triplet $(\mathbb H^n, d_0, dx)$ is a $Q-$Ahlfors regular metric measure space, with $Q=2n+2$ denoting the Hausdorff dimension of the group.  
 
 The  differential structure,  can be described in terms of a $g_0-$orthonormal frame $X_1,...,X_{2n}$ of $H$, and a 
 vector $Z\in V$, such that $[X_i, X_{i+n}]=Z$ for $i=1,..,n$ and all other commutators vanish. In exponential coordinates $(x_1,...,x_{2n},x_{2n+1})$, one can write these vector fields as $Z=\p_{x_{2n+1}}$,
 $$X_i= \p_{x_i} + \frac{x_{i+n}}{2}\p_{x_{2n+1}} \text{ for } i=1,...,n \quad \text{ and } X_i= \p_{x_i} - \frac{x_{i-n}}{2}\p_{x_{2n+1}} \text{ for } i=n+1,...,2n .$$
 For any domain $\Omega\subset \mathbb H^n$, and $1< p <\infty$, the horizontal  $p-$energy of a Lipschtiz function $u:\Omega\to \R$ is 
 $$E_p(u,\Omega)=\int_\Omega|\nabla_0 u|^p= \int_\Omega \bigg(\sum_{i=1}^{2n} (X_i u)^2 \bigg)^\frac{p}{2},$$
 where we have denoted by $\nabla_0 u$ the $2n-$dimensional vector field with components $X_1 u, ...., X_{2n} u$ .
 The subriemannian Heisenberg group admits a Poincar\'e inequality (see \cite{VCSC}, \cite{jerison})
 \begin{equation}\label{poincare}
 \int_B |u-u_B|  \le C\int_B |\nabla_0 u| ,
 \end{equation}
 where $B=B(x_0, R)=\{y \in \mathbb H^n \text{ such that }d_0(x_0,y)<R\}$ is a Carnot-Carath\'eodory metric ball centered in $x_0$ with radius $R>0$.
 The Ahlfors regularity and \eqref{poincare} yield as consequence all the necessary  embedding estimates for the horizontal Sobolev spaces
 $$HW^{1,p}(\Omega) = \{f\in L^p(\Omega) \text{ such that } |\nabla_0 f|\in L^p(\Omega)\},$$ with $1\le p <\infty$ and $\Omega\subset \mathbb H^n$. For more details we refer the reader to  \cite{HK,CDPT}, and here we only emphasize that membership in the horizontal Sobolev spaces does not guarantee the existence of weak derivatives along the center of the group. 

\bigskip

The Euler-Lagrange equation for the horizontal $p-$energy gives rise to the {\it subelliptic $p-$Laplacian }
\begin{equation}\label{plap} 
L_p u :=\sum_{i=1}^{2n} X_i  \Bigg( |\nabla_0 u|^{p-2} X_i u \Bigg)=0. \end{equation}

In this survey we will also necessarily consider a more general class of equations based on \eqref{plap},

\begin{equation}\label{maineq}
\sum_{i=1}^{2n} X_i A_i(\nabla_0 u)=0 
\text{ in }\Omega,
\end{equation}
where there exist $\ddelta\ge 0$, $\lambda',\Lambda'>0$ such that for any $\xi, \eta \in \R^{2n}$
\begin{equation}\label{structure-zero}
\begin{cases}
 \lambda' (\ddelta+|\xi|^2)^{\frac{p-2}{2}} |\eta|^2 \le  \sum_{i,j=1}^{2n} \p_{\xi_j} A_i(\xi) \eta_i \eta_j \le \Lambda' (\ddelta+|\xi|^2)^{\frac{p-2}{2}} |\eta|^2,
 \\ 
|A_i(\xi)| +  |\p_{x_j} A_i(\xi)| \le  \Lambda' |\xi|^{p-1}.
\end{cases}
\end{equation}
We will also consider weak solutions of the parabolic counterpart of \eqref{maineq},
\begin{equation}\label{maineq-zero}
\p_t u = \sum_{i=1}^{2n} X_i A_i(\nabla_0 u),
\end{equation}
in a cylinder $Q=\Om\times(0,T)$.

Both \eqref{maineq} and \eqref{maineq-zero} need to be interpreted in a weak sense. Specifically, a function $u\in HW^{1,p}_{loc} (\Omega)$ is a weak solution of \eqref{maineq} if for every test function  $\phi \in C^\infty_0 (\Omega)$ one has 
$$\int_\Omega A_i(\nabla_0 u)X_i \phi =0.$$
Likewise, for $T>0$, a function $u \in L^p([0,T], HW^{1,p}_{loc} (\Omega))$ is a weak solution of \eqref{maineq-zero} if for every test function  $\phi \in C^\infty_0 (\Omega \times [0,T])$ one has 
$$ \int_0^T \int_{\Omega} u \p_t \phi   = \int_0^T \int_\Omega A_i(\nabla_0 u)X_i \phi .$$
Thanks to the doubling property of the Haar measure and the Poincar\'e inequality, positive weak solutions of \eqref{maineq-zero} and of \eqref{maineq} satisfy a Harnack inequality, and consequently are H\"older continuous (see \cite{ ACCN}). As in the Euclidean case, the H\"older regularity of the horizontal gradient  of solutions $\nabla_0 u$ is more involved. In addition to the degeneracy of ellipticity which comes from the presence of the nonlinear term, one has a novel challenge which is purely non-Euclidean: Formally differentiating \eqref{maineq} along a given horizontal vector field $X_k$ and setting $w=X_k u$, yields
\begin{equation}\label{formal-diff}
\sum_{i=1,j}^{2n} X_i A_{i,\xi_j}(\nabla_0 u) X_j w + \sum_{i=1}^{2n} [X_k,X_i] A_i (\nabla_0 u) +  \sum_{i=1,j}^{2n} X_i A_{i,\xi_j}(\nabla_0 u) [X_k, X_j] u =0 .
\end{equation}
The non-commutativity of the horizontal vector fields gives rise to terms containing derivatives of the weak solution along the center of the group, i.e. $Zu$, which in principle may not even exist. 
In the literature there are essentially two methods to  address this new difficulty and make this informal derivation rigorous: The first is to use fractional differential quotients, coupled with H\"ormander observation that differentiability along the horizontal directions yield essentially the existence of a half derivative along the commutators (see \cite{Ho,Ca1}. The second method is that of {\it vanishing viscosity solutions}\footnote{Although the limit of such solutions is indeed a viscosity solution of the limit equation, aside from the name our work does not consider at all viscosity solutions}, and it consists in approximating the equation with a regularized version, for which the Euclidean regularity applies. This is done through 
 an approximation of the subriemannian structure $(\mathbb H^n, g_0)$ with a sequence of Riemannian metrics $g_\e$ so that, when $\mathbb H^n$ is endowed with the corresponding distance functions $d_\e$, it  converges in the Gromov-Hausdorff sense to the metric space $(\mathbb H^n, d_0)$. One can show that the critical points of the $g_\e-$energies 
$$\int_{\Om} |\nabla_\e u|^p $$
converge in a suitable sense to $p-$harmonic functions for the subelliptic p-Laplacian. Since the Euclidean theory applies when $\e>0$, one then can differentiate the approximating equation without worrying about the derivatives in the direction of the center. This second approach allows for a more streamlined exposition and we will use it to describe the whole development of the regularity theory in this note. It is  is described in detail in Section \ref{riemannian-approx}.

\bigskip
The current status for the interior regularity theory can be summarized in the following theorems.

\begin{thrm}[Stationary case]\label{main-th-stationary}
Let $A_i$ satisfy the structure conditions \eqref{structure-zero} with $\delta\ge 0$, and $1<p<\infty$.
If $u\in  HW_{\loc}^{1,p}(\Om))$ is a weak solution of \eqref{maineq}  in $\Om$, then for every metric ball $B=B(x_0,r)$ with  $3B\subset  \Om$, there exist constants $\alpha\in (0,1)$, and $C=C(n,p,\lambda, \Lambda)>0$ depending on the structure conditions, such that 
\[
||\nabla_0 u||_{C^{\alpha}(B)} \le C r^{-\alpha-Q/p} \bigg(\int_{2B}(\delta+|\nabla_0u|^2)^{\frac{p}{2}} \bigg)^{\frac{1}{p}},
\]
where the H\"older regularity is measured in terms of the Carnot-Carath\'eodory distance $d_0$.
\end{thrm}

To state the pertinent results in the parabolic regime we define the cylinders
$$ Q_{\mu, r} : = B(x, r) \times [t_0 - \mu r, t_0]. $$
The introduction of these anisotropic balls is due to DiBenedetto \cite{DB} in the Euclidean setting.  In these
 non homogeneous spheres,  the ratio between the 
spatial and temporal dimension 
compensates for the lack of homogeneity of the equation.

\begin{thrm}[Parabolic case]\label{main-th-parabolic} 
Let $A_i$ satisfy the structure conditions \eqref{structure-zero} with $\delta=0$, and 
let $u\in L^p((0,T), HW_{\loc}^{1,p}(\Om))$ be a weak solution of \eqref{maineq-zero}  in $Q=\Om\times (0,T)$.  If $\ 2\le p \le 4$ then $|\nabla_0 u| \in L^\infty_{\loc}(Q)$ and $\partial_t u, Zu \in L^q_{\loc}(Q)$ for every $1\le q<\infty$. Moreover, one has
that  for any $Q_{\mu, 2r}\subset Q$, in the range $2<p\le 4$ one has 
\begin{equation}\label{desired-0}
\sup_{Q_{\mu, r}}  |\nabla_0u|
\le C  \max\Big( \Big(\frac{1}{\mu r^{N+2}}\int\int_{Q_{\mu,2r}} |\nabla_0u|^p  \Big)^{\frac{1}{2}}, \mu^{\frac{p}{2(2-p)}}\Big),
\end{equation}
where $C=C(n,p,\lambda, \Lambda, r,\mu)>0$. Note that the special case $p=2$ has been studied before using linear techniques, and the estimates we obtain in that case are similar to \eqref{desired-0}, with the anisotropic cylinders being replaced with the classic parabolic cylinders. 
\end{thrm}

In the special case where $\delta>0$ and the vanishing of the gradient does not present any degeneracy one has a stronger regularity result
\begin{thrm}[Non-degenerate parabolic case]\label{main1}  
Let $A_i$ satisfy the structure conditions \eqref{structure-zero} for some $p\ge 2$ and $\ddelta>0$. 
Let $u\in L^p((0,T), HW_0^{1,p}(\Om))$ be a weak solution of \eqref{maineq-zero}  in $Q=\Om\times (0,T)$.
 For any metric ball $B=B(x,r)\subset \subset \Om$ and $T>t_2\ge t_1\ge 0$, 
there exist  constants $C=C(n, p, \lambda, \Lambda, d(B, \p\Om), T-t_2, \ddelta)>0$  and 
$\al=\al (n, p, \lambda, \Lambda, d(B, \p\Om), T-t_2, \ddelta, r) \in (0,1)$ such that
\begin{equation}\label{c1alpha}
||\nabla_0 u ||_{C^{\alpha} (B\times (t_1,t_2) )} + ||Zu||_{C^{\alpha} (B\times (t_1,t_2) )}\le C
\bigg(\int_0^T \int_\Om (\ddelta+|\nabla_0 u|^2)^{\frac{p}{2}}  \bigg)^{\frac{1}{p}} .
\end{equation}
Moreover, if $A_i$ are smooth, then the solution is smooth as well.
\end{thrm}

The three theorems above continue to hold for more general equations of the form 
\begin{equation}\label{xdependence}\p_t u - \sum_{i=1}^{2n} X_i A_i (x, \nabla_0 u)=0,\end{equation}
with $A_i$ satisfying structure conditions similar to \eqref{structure-zero}, see \cite{CCLDO,CCG, CCZ,Muk2}.
In particular, thanks to Darboux theorem, the regularity continues to hold in the setting of subriemannian contact manifolds, see \cite{CCLDO}.
We also note that in the special case where there is no direct dependence on the space variable, i.e.  $A_i(x,\xi)=A_i(\xi)$, the parameters dependence  in all theorems  is more explicit, with
$C=C(n,p,\lambda, \Lambda) \mu^{\frac 1 2}>0$. 

In the linear setting, the first results on interior regularity of the gradient of  weak solutions  of subelliptic equations go back to the pioneering work of  H\"ormander \cite{Ho}. In the nonlinear case,  an early result on  the H\"older regularity of the  horizontal gradient of weak solutions of equation \eqref{maineq} is due to the first named author, who in the special quadratic growth case $p=2$ proved smoothness of solutions (see \cite{Ca1, Ca2} in the Heisenberg group and later in any Carnot group). Subsequently several authors contributed to the study of the regularity. 
Domokos and Manfredi \cite{DM1, DM2} proved the H\"older continuity of the gradient of $p$-harmonic functions by the Cordes perturbation techniques when $p$ is close to 2 . No explicit bound on $p$ was given. Domokos \cite{Dom} showed that the vertical derivative 
$Zu \in L^p_\loc(\Omega)$, if $1 < p < 4$, for the weak solutions $u$ of equation (\ref{maineq}). The Lipschitz continuity of solutions to equation (\ref{maineq}) with $\delta\ge 0$ was proved in \cite{Mingione-Zatorska-Zhong} for $2\le p<4$. This extended an earlier result of
Manfredi and Mingione \cite{Manfredi-Mingione} for the non-degenerate case $\delta> 0$ with $p$ in a smaller range. Both of the proofs in \cite{Manfredi-Mingione, Mingione-Zatorska-Zhong} were based on Domokos' integrability result on $Zu$.
Eventually 
Theorem \ref{main-th-stationary} was proved by the third named author in the range $2\le p <\infty$ and then extended to the full range $1<p<\infty$ in \cite{MZ}. We also mention  alternative proofs in \cite{CCLDO,Ricciotti}.

Theorem \ref{main1} was proved by Garofalo and the first two named authors in \cite{CCG}. In that paper the authors also provide a complete Schauder theory for \eqref{xdependence} in the non-degenerate case with $\delta>0$.

Theorem \ref{main-th-parabolic}  is a recent result of the authors from  \cite{CCZ}. In that paper, in addition to the Lipschtiz regularity for solutions in the range $2\le p \le 4$, the authors also improve on the structural hypotheses in \cite{CCG}, proving that 
 the conditions \eqref{structure-zero} are sufficient to construct suitable Riemannian approximations (which instead were part of the hypotheses in \cite{CCG}). 


\bigskip

In the following we will endeavour to sketch a road map of the various ideas and techniques that have led to these results, along with a detailed list of references. We will also highlight the challenges and conclude with a short list of open problems. 

\section{approximation schemes}

\subsection{Subriemannian approximation approach}\label{riemannian-approx}

Subriemannian metrics can be seen as Gromov-Hausdorff limits of sequences of collapsing Riemannian metrics, where the non-horizontal directions are increasing penalized. This metric approximation in turns leads to an approximation of those subelliptic operators naturally arising from the concepts of energy, like the sub-Laplacian or the subelliptic $p-$Laplacian, by means of their Riemannian counterparts.  This approximation is in fact a regularization scheme, and it has been used widely in the literature to prove a-priori estimates for weak solutions of subelliptic PDE. 

In this section we will look at the specific case of the Heisenberg group.
Consider the family of left-invariant Riemannian metrics $g^\e$ in $\Hn$ parametrized by $\e>0$, defined by setting the frame $\{X_1,..., X_{2n}, \e X_{2n+1}\}$ to be $g^\e-$orthonormal.  The $g_\e$ gradient $\nabla_\e u $ of a function $u$ has then length $$|\nabla_\e u|^2_\e= \sum_{i=1}^{2n} (X_i u)^2 + \e^2 (X_{2n+1}u)^2.$$
 In order to study the balls associated to the associated control metric $d_\e$ one can consider a regularized gauge function
\begin{equation}\label{gauge}
N_\e^2(x)= \sum_{i=1}^{2n} x_i^2 + \min \bigg\{ |x_{2n+1}|, \frac{x_{2n+1}^2}{\e^2} \bigg\}.
\end{equation}
For $x,y\in \Hn$, we will denote by $y^{-1}$ the inverse of an element and $xy$ the product, with respect to the Heisenberg group law.   In \cite[Lemma 2.13]{CC} it is proved that the pseudo distance function $d_{G,\e}(x,y):=N_\e(y^{-1}x)$  is equivalent to the distance function $d_\e$, i.e. there exists a constant $A>0$ such that for all $x,y\in \Hn$, and for all $\e>0$ one has
\begin{equation}\label{ballsequiv}
A^{-1} d_{G,\e}(x,y) \le d_\e (x,y) \le d_{G,\e} (x,y).
\end{equation}
This estimate expresses in a quantitative fashion the fact that at short scale, with distance less than $\e>0$, the metric $d_\e$ is essentially the Euclidean metric, while at larger scales it behaves like a subriemannian metric. 

From \eqref{ballsequiv} it follows immediately that the doubling property 
$$|B_\e(x,2R)| \le C |B_\e(x,R)|,$$
holds uniformly in $\e>0$ with $C>1$ independent of the choice of $\e>0$. In \cite{CC, DMR} it is also shown that the Poincar\'e inequality holds uniformly with constants independent of $\e>0$. 

For $1<p<\infty$,  and $\Om \subset \Hn$, the energy functionals associated to $g_\e$ are
$$E_p(u,\Om)=\int_\Om |\nabla_\e u|^p  $$
and their Euler-Lagrange equations give rise to regularized $p-$Laplacians
$$L_{p,\e} u:=X_i^\e \Bigg( |\nabla_\e u|^{p-2} X_i^\e u \Bigg),$$
where $X_i^\e= X_i$ for $i=1,..., 2n$ and $X_{2n+1}^\e= \e X_{2n+1}=\e \partial_{x_{2n+1}}$.

The Euclidean and Riemannian elliptic and parabolic theory \cite{Le, Tolksdorff, DB,DF} can then be invoked to provide interior $C^{1,\alpha}$ regularity for the solutions of $$\p_t u^\e(x,t) = L_{p,\e} u^\e(x,t) \text{ in }\Omega \times (0,T)$$
or its stationary counterpart $L_{p,\e}u^\e=0$.

The approach to regularity in \cite{CC}, \cite{CCLDO}, \cite{CCG}, \cite{CCZ} consists in proving regularity estimates that are stable as $\e\to 0$.   Note that stable estimates for the interior $C^\alpha$ regularity of solutions $u^\e$ follow from the stability of the doubling property and of the Poincar\'e estimates (see \cite{CC}  where in addition one can find also stable estimates for the heat kernels). 

In order to prove $C^{1,\alpha}$ estimates near a point $x\in \Omega$ for weak solutions $u$ of \eqref{maineq}, one considers  for each $\e>0$ a sequence of smooth solutions to the regularized Dirichlet problem
\begin{equation}
\begin{cases} \sum_{i=1}^{2n+1} X_i^\e \Bigg([\delta+ |\nabla_\e u_\e|^2]^{\frac{p-2}{2}} X_i^\e u_\e \Bigg)=0 \text{ in } B(x,r)\subset \subset \Om \\ u_\e = u \text{ on } \partial B(x,r).
\end{cases}
\end{equation}

Uniform convergence $u^\e \to u$ in compacts subsets of $B(x,r)$ follows from the stable estimates on the  $C^\alpha$ regularity of solutions $u^\e$, as mentioned above.

\begin{rmrk}This procedure has been extended in \cite{CCZ} to approximate the quasilinear equation \eqref{maineq}, leading to define 
a family of functions $A_{i,\delta}^\e(\xi)$  defined for every $\xi \in \R^{2n+1}$, 
which approximate the functions 
$A_{i}$ in the following sense. 
If $\xi =  \sum_{i=1}^{2n} \xi_i X_i^\e + \xi_{2n+1} X_{2n+1}^\e$ , and $\xi^\e =  \sum_{i=1}^{2n} \xi_i X_i^\e + \e \xi_{2n+1} X_{2n+1}^\e$   one has 
\begin{equation}\label{structure2}
A_{i,\delta}^\e(\xi^\e)\ \underset{\e\to 0^+, \delta \to 0}{\longrightarrow}\    A_{i}(\xi_1,...,\x_{2n}, 0),
\end{equation}
and 
\begin{equation}\label{structure-epsilon}
\begin{cases}
\lambda (\ddelta+|\xi|^2)^{\frac{p-2}{2}} |\eta|^2 \le  \p_{\xi_j} A_{i,\delta}^\e(\xi) \eta_i \eta_j \le \Lambda (\ddelta+|\xi|^2)^{\frac{p-2}{2}} |\eta|^2, 
\\
|A_{i,\delta}^\e(\xi)| +  |\p_{x_j} A_{i,\delta}^\e(\xi)| \le  \Lambda (\ddelta+|\xi|^2)^{\frac{p-1}{2}},
\end{cases}
\end{equation}
 for all $\eta\in \R^{2n+1}$, 
 for some $0<\lambda\le \Lambda <\infty$ independent of $\e$.
\end{rmrk}
In this more general setting, the approximation is based on the study of the equations 
\begin{equation}\label{eqapprox}
\begin{cases}   \sum_{i=1}^{2n+1}X_i^\e (A_{i,\delta}^\e(\nabla _\epsilon u^\epsilon)) =0 \text{ in } B(x,r)\subset \subset \Om \\ u_\e = u \text{ on } \partial B(x,r).
\end{cases}
\end{equation}
and the parabolic counterpart

\begin{equation}\label{approx1}
\begin{cases}  \p_t u^\epsilon=  \sum_{i=1}^{2n+1}X_i^\e (A_{i,\delta}^\e(\nabla _\epsilon u^\epsilon))  \text{ in } B(x,r)\times (t_1,t_2)  \\ u_\e = u \text{ on } \partial_P \bigg(B(x,r)\times (t_1,t_2)\bigg),
\end{cases}
\end{equation}
where by $\partial_P (B\times (t_1,t_2))$ we denote the parabolic boundary of the cylinder, i.e. $\partial B \times (t_1,t_2)\cup B\times \{t_1\}.$
For both equations, estimates of the gradient of the solutions $(u_\epsilon)$ uniformly in the variable $\epsilon$ provide the regularity of the limit function. The heart of the matter is then to prove higher regularity estimates for the $u_\e$ which are uniform in both parameters $\e$ and $\delta$.

\subsection{Hilbert-Haar approach}

In the setting of Euclidean spaces, the
Hilbert-Haar theory gives the existence of Lipschitz continuous minimizers for convex
functional if the boundary values satisfy the bounded slope condition. 
The smooth boundary values in strictly convex domain satisfy this condition. We refer to \cite{Clark} for the detailed discussions in this setting.  

In the setting of Heisenberg group, we have similar results as those in Euclidean spaces. Let $\Omega\subset \mathbb{H}^n$ be a bounded domain.  We consider it as a  domain in $\mathbb{R}^{2n+1}$. The following theorem was proved in \cite{Zhong}, where we need an extra condition on $A=(A_1, ...,A_{2n})$ of equation (\ref{maineq}): $A=\nabla f$ for a function $f:\mathbb{R}^{2n}\to \mathbb{R}$.

\begin{thrm} Suppose that $\Omega\subset\mathbb{R}^{2n+1}$ be a bounded uniformly convex domain. Let $\phi$ be a $C^2$ function
in $\overline{\Omega}$, and $ u\in HW^{1,p}(\Omega)$ be the unique solution of equation (\ref{maineq}) with 
$u-\phi\in HW^{1,p}_0(\Omega)$. Then $u$ is Lipschitz continuous in $\Omega$, that is, 
\[
\| \nabla_0 u\|_{L^\infty(\Omega)}\le M,
\]
where the constant  $M$ depends on $n, \Omega$ and $\max_{\overline{\Omega}}( |D\phi| +|D^2\phi|)$.
\end{thrm}

In the theorem, $D\phi$ is the Euclidean gradient of $\phi$ and $D^2\phi$ the Hessian matrix of $\phi$. The proof of the above theorem
is similar to that in the Euclidean setting. The essential point is that all linear functions
\[
L(x)=\sum_{i=1}^{2n+1} a_ix_i, \quad x=(x_1,x_2,...,x_{2n+1}), a_i\in \mathbb{R},
\]
are solutions to equation (\ref{maineq}) with its own boundary value. This is the case for the Heisenberg group. It was shown in
\cite{DoM2} that this is also the case for the Carnot groups
of step two, Goursat groups and all groups of step three with linearly independent third order commutators. So, the above theorem is also true for these groups. 

A natural open problem arises: do we have similar theorem as the above one for the general Carnot groups? To answer this question, or to establish the Hilbert-Haar theory in the subriemannian setting, one can not use the Euclidean approach and one probably needs to built up the "subriemannian linear functions". 

The Hilber-Haar theorem provides a qualitative upper bound for the Lipschtiz norm of the weak solutions of \eqref{maineq}, and consequently allows for an approximation scheme where 
one considers solutions $u^\delta$ of \eqref{maineq} as $\delta\to 0$. In view of the results in the quadratic case $p=2$ in \cite{Ca1}, these are smooth functions, and so one can differentiate the equation without worrying about issues of differentiability along the center. 

\section{Cacciopoli type inequalities satisfied by the first derivatives}

The first step to prove higher regularity estimates for the $u_\e$ is to differentiate the equation. A direct computation ensures that, if $u^\epsilon$ is a solution of \eqref{eqapprox}, $w^\epsilon _l = X^\epsilon_l u^\epsilon$, with $ l = 1, 2, . . . , 2n$, and $s_l = (-1)^{[l/n]}$ then the function $w^\epsilon _l$ is  a  solution of

\begin{equation}\label{eqderivX}
 \sum_{i,j=1}^{2n} X^\epsilon_i\Big( A^\epsilon_{i, \delta, \xi_j}  (\nabla_\epsilon u^\epsilon) X^\epsilon _l X^\epsilon _ju^\epsilon \Big) + s_l Z (A^\epsilon _{l+s_ln \delta}  (\nabla_\epsilon u^\epsilon)) =0.\end{equation}

Using this equation, we obtain the following Cacciopoli type inequality. Its proof is standard. We  use $\varphi=
\eta^2X^\e_l u^\e$ as a testing function to equation (\ref{eqderivX})  to obtain the result.

\begin{lemma}\label{lemma:Cacci2}
Let $1<p<\infty$ and $u^\e$ be a weak solution of equation \eqref{eqapprox} with $\delta\ge0$. Then for any $\beta\ge 0$ and every $\eta\in C^\infty_0(\Omega)$, we have that
\begin{equation}\label{cacci2}
\begin{aligned}
\int_\Omega \eta^2 (\delta+ |\nabla_\e u^\e|^2)^{\frac{p-2+\beta}{2}} |\nabla_\e^2u^\e|^2\le  & c \int_\Omega \big( |\nabla_\e\eta|^2+|\eta||Z\eta|\big)(\delta+ |\nabla_\e u^\e|^2)^{\frac{p+\beta}{2}}\\
&+
c (\beta+1)^4\int_\Omega \eta^2 (\delta+ |\nabla_\e u^\e|^2)^{\frac{p-2+\beta}{2}} |Zu^\e|^2,
\end{aligned}
\end{equation}
where $c=c(n,p)>0$.
\end{lemma}

The main obstacle in establishing the regularity properties of the first derivatives for the gradient of the solution, is due to the presence of the derivative $Z(u^\epsilon)$ in the second member of 
\eqref{cacci2}, where $Z$ is a derivative of order 2. For this reason, also derivatives with respect to $Z$ have to be considered. 

In particular the function
$Zu^\epsilon$  is   a  solution of the equation
\begin{equation}\label{eqderivZ}
\sum_{i,j=1}^{2n}
X^\epsilon_i(A^\epsilon_{\delta i, \xi_j}(\nabla_\epsilon u^\epsilon )X^\epsilon _j Zu^\epsilon)=0
\end{equation}

The equation satisfied by $Zu^\epsilon$ can be interpreted as a linear equation, with coefficients depending on $\nabla_\e u^\e$. For this reason, 
a standard Caccioppoli type estimates is satisfied by the derivative $Zu^\e$.

\begin{lemma}[Lemma 3.4 \cite{CCG}]\label{stimaZu} Let $u^\e$ be a solution of \eqref{eqapprox}  with $\delta>0$.
For every $ \beta \geq 0$ and  non-negative $\eta\in C^\infty_0 (\Omega)$  one has
\begin{equation}
\begin{aligned}
\int_\Omega 
(\delta+|\nabla_\e u^\e|^2)^{\frac{p-2}{2}}|Zu^\e|^{\beta} | \nabla_\e Zu^\e|^2\eta^{4+\beta}
\le & C\int_\Omega
(\delta+|\nabla_\e u^\e|^2)^{\frac{p-2}{2}}|Zu^\e|^{\beta+2}|\nabla_\e\eta|^2 \eta^{2+\beta}
\end{aligned}
\end{equation}
where $C=C(\lambda, \Lambda)>0$.
\end{lemma}

The main goal is now  to remove the dependence on $Zu^\e$ in the estimate of  $X^\e_lu^\e$.

\section{The linear growth case}

The first higher regularity results for subelliptic quasilinear equations \cite{Ca1}, dealt with equations of the form  \eqref{maineq}
with $A_i$ differentiable and satisfying the structure assumptions
\eqref{structure-zero}
and $|A_i(\eta)|\le C(1+|\eta|)$. In \cite{Ca1} the first named author proved smoothness of weak solutions, using fractional derivatives estimates inspired by H\"ormander's work in \cite{Ho}. The proof in  \cite{Ca1}  was made without using the Riemannian approximation.  More precisely, instead of Cacciopoli type inequalities for the regularized gradient, the argument in \cite{Ca1} yield Cacciopoli type inequalities for the difference quotients, with constants uniform in the increment. As a consequence, one  obtains estimates of the derivatives in the limit when the increment goes to $0$. This result in turn was then used in a regularization scheme by third named author in \cite{Zhong} to establish $C^{1,\alpha}$ estimates for $p-$harmonic functions for $p>1$. In order to simplify the presentation, we use the same technique for all values of $p$ and present here a sketch of a new proof of the regularity for $p=2$, based on the Riemannian approximation technique and without fractional derivatives.

A direct consequence of Lemma  \ref{cacci2} with $p=2$ and $\beta=0$ is the estimate 
of the $L^2$ norm of $Z u^\e$. Indeed 
$Z u^\e = (X^\e_1 X^\e_{n+1} - X^\e_{n+1} X^\e_{1})u^\e $, so that, first integrating by parts,  and then using the Lemma we obtain for every $\alpha>0$ a constant $C_\alpha$ such that
$$
\int_\Omega |Z u^\e|^2 \phi^2   \leq 2 \int_\Omega |\nabla_\e u^\e|  (|\nabla_\e Z u^\e|  \phi^2 +2 \phi |\nabla_\e \phi| |Z u^\e| )  \leq C_\alpha \int_{supp(\phi)} |\nabla_\e u^\e|^2 + \alpha \int_\Omega |Z u^\e|^2 \phi^2|\nabla_\e \phi|^2 
$$
Choosing $\alpha$ sufficiently small, we obtain the estimate 
$|| Z  u^\e ||_{L^2} \le C ||\nabla_\e u^\e||_{L^2}.$
From the latter and invoking the  Caccioppoli inequality in  Lemma \ref{lemma:Cacci2} with $p=2$, $\beta=0$, a standard argument involving Morrey classes would then yield the H\"older regularity of $\nabla_0 u^\e$ with estimates uniform in $\e>0$ and in view of the Schauder theory (see \cite{Xu}) conclude the smoothness of $u^e$, with estimates uniform as $\e\to 0$ (see \cite{Ca1} for details on the adaptation of Morrey's estimates and for further references).

\section{Lipschitz regularity of weak solutions of equation \eqref{maineq}}

In the section, we discuss the local Lipschitz continuity of weak solutions to equation ({\ref{plap}) for the full range $1<p<\infty$. Actually, as shown in \cite{Zhong}, this result holds for weak solutions to equation (\ref{maineq}), see also \cite{Muk1} for more general equations. To simplify the presentation, we will work on the $\delta$-regularized $p$-Laplacian equation,
\begin{equation}\label{maineqdelta}
\sum_{i=1}^{2n} X_i\big( (\delta+|\nabla_0 u|^2)^{\frac{p-2}{2}} X_iu\big)=0,\end{equation}
where $\delta>0$. The following theorem gives us the Lipschitz regularity for equation (\ref{maineqdelta}) with $\delta\ge 0$.

\begin{thrm}\label{thm:Lip} Let $1<p<\infty$ and $u\in HW^{1,p}(\Omega)$ be a weak solution of equation (\ref{maineqdelta}) with $\delta\ge0$. Then $\nabla_0u\in L^\infty_{\loc}(\Omega; \mathbb{R}^{2n})$. Moreover, for any ball $B_{2r}\subset \Omega$, we have that
\begin{equation}\label{nablasup}
\sup_{B_r} |\nabla_0u|\le \frac{c}{r^{Q/p}}\Big(\int_{B_{2r}} (\delta+|\nabla_0 u|^2)^{\frac p 2}\Big)^{1/p},
\end{equation}
where $c=c(n,p)>0$.
\end{thrm}

In the above theorem, we state the Lipschitz regularity for equation (\ref{maineqdelta}) for all $\delta\ge 0$. Actually, we prove the theorem for the case $\delta>0$ and let $\delta$ go to zero to obtain the result for the case $\delta=0$. Notice that the constant $c$ in estimate (\ref{nablasup})
does not depend on $\delta$. This is also the case for all of the constants in the estimates in the remaining of this section. We first obtain the estimates for $\delta>0$ and then that for $\delta=0$ by letting $\delta$ go to zero. 
\subsection{Regularization procedure}

Let $u\in HW^{1,p}(\Omega)$ be a weak solution 
of equation (\ref{maineqdelta}) with $\delta>0$. 
The existing approach to prove the Lipschitz continuity of $u$, even in the Riemannian setting, involves differentiating equation (\ref{maineqdelta}) to consider  the equations for $X_lu, l=1,2,...,2n,$ and $Zu$.

\begin{rmrk}The original result of \cite{Zhong}  is based on the Hilbert-Haar theory recalled in section 2.1 as in \cite{Zhong}. 
 Let $u\in HW^{1,p}(\Omega)$ be a weak solution 
of equation (\ref{maineqdelta}) with $\delta>0$.
By the Hilbert-Haar theory in section 2.1, we may assume that $\nabla_0u\in L^\infty(\Omega;\mathbb{R}^{2n})$. With this assumption, equation (\ref{maineqdelta}) is uniformly elliptic, since we assume that $\delta>0$. 
We may now apply
the results in \cite{Ca1} and recalled in section 3 here. From Theorem 1.1 and Theorem 3.1 of \cite{Ca1}, we know that $\nabla_0u$ and $Zu$ are H\"older continuous in $\Omega$ and that
\[
\nabla_0u\in HW^{1,2}_{\loc}(\Omega;\mathbb{R}^{2n}), \quad Zu \in HW^{1,2}_{\loc}(\Omega;\mathbb{R}^{2n}).
\]
We explicitly note that the original proof of \cite{Ca1} was not based on Riemannian approximation, which could  be completely avoided by using fractional difference quotients.
\end{rmrk}
Since we based the presentation of the quadratic growth case on Riemannian approximation, we will use the same technique here. As in section 2.2, we consider for each $\varepsilon>0$  to the subriemannian approximation equation of (\ref{maineqdelta})
\begin{equation}\label{maineqdeltaep}
\begin{cases}
\sum_{i=1}^{2n+1} X_i^\e \Big( (\delta+|\nabla_\e u^\e|^2)^{\frac{p-2}{2}} X_i^\e u^\e\Big) =0 \quad &\text { in }\Omega;\\
u^\e =u &\text{ on } \partial \Omega.
\end{cases}
\end{equation}

We have already noted that the  $u^\e$ belongs to $C^\infty(\Omega)$, and obtained  Caccioppoli inequality for $X^\e u^\e$ in Lemma \ref{lemma:Cacci2}. 
Here we will obtain estimates uniform in $\e$. Eventually, we let $\e$ go to zero to obtain the estimates for $u$ and we prove the Lipschitz continuity of $u$. 
\subsection{An Homogeneous Cacciopoli type inequality for the horizontal derivatives}
We have already noted in section 3 that the main step in the regularization procedure is to remove the expression of $Zu^\e$ from the second member of the Cacciopoli inequality Lemma \ref{lemma:Cacci2}.

The main step in order to obtain 
Theorem \ref{thm:Lip} is to prove following surprising Caccioppoli
type inequality for $\nabla_\e u^\e$, 
where the expression $Zu^\e$ does not appear any more. This inequality is similar to that for weak solutions of $p$-Laplacian equation in the setting of Euclidean spaces, and the estimate of the gradient follows from the Moser procedure.

\begin{thrm}\label{thm:Cacci1} Let $1<p<\infty$ and $u^\e$ be a weak solution of equation (\ref{maineqdeltaep}) with $\delta\ge 0$. Then 
for any $\beta\ge 2$ and every $\eta\in C^\infty_0(\Omega)$, we have that
\begin{equation}\label{cacci1}
\int_\Omega \eta^2 (\delta+ |\nabla_\e u^\e|^2)^{\frac{p-2+\beta}{2}}|\nabla_\e^2u^\e|^2 \le c\beta^{10}\big(||\nabla_\e\eta||^2_{L^\infty}+||\eta Z\eta||_{L^\infty}\big)
\int_{spt(\eta)} (\delta+ |\nabla_\e u^\e|^2)^{\frac{p+\beta}{2}},
\end{equation}
where $c=c(n,p)>0$ and $spt(\eta)$ is the support of $\eta$.
\end{thrm}

In the special case $2\leq p \leq 4$ we will first study a Cacciopoli type inequality for the vector field $Zu^\e$ alone, and plug in the result in the previous equaiton. 
In the general case a much delicate procedure involving all the derivatives is needed
\subsection{The $2\le p \le 4$ case}

We have already seen in Lemma \ref{stimaZu} that $Zu^\e$ satisfies a Cacciopoli type inequality. The standard Moser technique, which ensures higher integrability and boundness of the solution, 
is composed by two steps: Caccioppoli inequality (which we have already established in Lemma \ref{stimaZu}), and Sobolev inequality. Here we replace the  Sobolev inequality with a simple interpolation inequality inspired from \cite{CLM}. This is where we use we make of the hypothesis $2\le p\le 4$ .

\begin{lemma}\label{SobXU} Assume that $2\le p \leq 4$ and let $u\in C^2(Q)$.
There exists a constant $C>0$ depending only on $n, p$ such that for every $ \beta \geq 0$ and non-negative $\eta\in  C^\infty_0 (\Om)$ we have
\begin{equation}\label{interpolation}
\begin{aligned}
\int_\Omega &
|Zu|^{p+\beta} \eta^{p+\beta} 
\le
C(p+\beta)\int_\Omega 
(\delta + |\nabla_0u|^2)^{(p-2)/2}
|Zu|^{\beta} | \nabla_0Zu|^2\eta^{4+\beta}  \\
&   + C(p+\beta) ||\nabla_0\eta||_{L^\infty}\int_{spt(\eta)} (\delta+ |\nabla_0u|^2)^{\frac{p+\beta}{2}}
\end{aligned}
\end{equation}

\end{lemma}
\begin{proof}
We provide here the simple proof which directly follows from the definition of $Z$. Indeed 
$
Zu=X_lX_{n+l}u-X_{n+l}X_lu, 
$ so that 
we can write 
\[ |Zu|^{p+\beta}=|Zu|^{p-2+\beta}Zu (X_lX_{n+l}u-X_{n+l}X_lu).
\]
Then integration by parts gives us
\begin{equation}\label{ML1}
\begin{aligned}
\int_\Omega 
|Zu|^{p+\beta} \eta^{p+\beta} =&\int_\Omega |Zu|^{p-2+\beta}Zu (X_lX_{n+l}u-X_{n+l}X_lu)\eta^{p+\beta}\\
=&-(p-1+\beta)\int_\Omega |Zu|^{p-2+\beta}(X_lZuX_{n+l}u-X_{n+l}ZuX_lu)\eta^{p+\beta}\\
&-(p+\beta)\int_\Omega|Zu|^{p-2+\beta}Zu(X_{n+l}uX_l\eta-X_luX_{n+l}\eta)\eta^{p-1+\beta}\\
\le & \, 2(p+\beta)\int_\Omega |\nabla_0u||Zu|^{p-2+\beta}|\nabla_0Zu|\eta^{p+\beta}\\
&+2(p+\beta)\int_\Omega |\nabla_0u||Zu|^{p-1+\beta}|\nabla_0\eta|\eta^{p-1+\beta}=I_1+I_2
\end{aligned}
\end{equation}
Applying  Young's  inequality $abc \leq \frac{1}{2}a^2 + \frac{4-p}{2(p+\beta)}b^{2(p+\beta)/(4-p)} +  \frac{2p-4+\beta}{2(p+\beta)} c^{2(p+\beta)/(2p-4+\beta)}$ to $I_1$ and 
$ab \leq \frac{1}{\beta + p}a^{\beta + p} + (\beta + p -1) b^{(\beta + p)/(\beta + p -1)} $
 to $I_2$ we immediately obtain the thesis. \end{proof}

We can now put together the Cacciopoli inequality \eqref{cacci1} and the modified Poincar\'e-like inequality for $Zu^\e$. 
Proceeding in this way, one obtains an estimate very different from a typical step in the standard Moser iteration, which provides a gain of integrability for the solution. More in detail, one can prove

\begin{lemma}\label{mainlemma}
Let $u^\e$ be a solution of \eqref{maineqdeltaep}  with $\delta>0$ and $2\leq p \leq 4$. 
Then for every $ \beta \geq 0$ and non-negative $\eta\in C^\infty_0 (\Om)$ , we have
\begin{equation}\label{mainestimate-0}
\begin{aligned}
\int_\Omega &
|Zu^\e|^{p+\beta} \eta^{p+\beta} 
\le  C(p+\beta)^{p+\beta} ||\nabla_\e\eta||^{p+\beta}_{L^\infty}\int_{spt(\eta)} (\delta+ |\nabla_\e u^\e|^2)^{\frac{p+\beta}{2}}
\end{aligned}
\end{equation}
\end{lemma}

We explicitly note that the inequality we have obtained can be interpreted as a gain of differentiability. In fact in the inequality one has an estimate of an integral of a derivative of order 2, that is  $Zu^\e$, in terms of an integral of   a first order derivative $|\nabla_\e u^\e|$.

Plugging this estimate in the Cacciopoli type inequality \eqref{lemma:Cacci2} we obtain 
Theorem \ref{thm:Cacci1} in the special case $2\leq p \leq4$

%
%
%
%
%
%
%
%
\subsection{The general case}

In order to handle the general case, the following
estimate for $Zu^\e$ is essential. This is a reverse  H\"older type inequality for $Zu^\e$, associated with $\nabla_\e u^\e$ and $\nabla_\e^2u^\e$. 

\begin{lemma}\label{lemma:rev}
Let $1<p<\infty$ and $u^\e$ be a weak solution of equation (\ref{maineqdeltaep}) with $\delta\ge0$.  Then for any $\beta\ge 2$ and every non-negative $\eta\in C^\infty_0(\Omega)$, we have that
\begin{equation}\label{cacci3}
\int_\Omega \eta^{\beta+2}  (\delta+ |\nabla_\e u^\e|^2)^{\frac{p-2}{2}}  |Zu^\e|^\beta |\nabla_\e^2u^\e|^2\le c\beta^2 ||\nabla_\e\eta||^2_{L^\infty} \int_\Omega\eta^\beta  (\delta+ |\nabla_\e u^\e|^2)^{\frac{p}{2}} 
|Zu^\e|^{\beta-2} |\nabla_\e^2u^\e|^2,
\end{equation}
where $c=c(n,p)>0$.
\end{lemma}

The proof of Lemma \ref{lemma:rev} involves a non-standard testing function $\varphi=\eta^{\beta+2} |Zu^\e|^\beta X^\e_lu^\e$ to equation (\ref{eqderivX})
or (\ref{eqderivZ}) for all $l=1,2,..., 2n$ and sum up the estimates. One crucial point of the proof is that $Zu^\e$ is a weak solution to equation (\ref{eqderivZ}).
Actually this is the reason that $Zu^\e$ enjoys this type of reverse inequality. 

An immediate consequence of (\ref{cacci3}) by H\"older's inequality is the following estimate

\begin{equation}\label{cacci4}
\int_\Omega \eta^{\beta+2}  (\delta+ |\nabla_\e u^\e|^2)^{\frac{p-2}{2}}  |Zu^\e|^\beta |\nabla_\e^2u^\e|^2\le c^\beta\beta^\beta ||\nabla_\e\eta||^\beta_{L^\infty} \int_\Omega\eta^\beta  (\delta+ |\nabla_\e u^\e|^2)^{\frac{p-2+\beta}{2}} 
 |\nabla_\e^2u^\e|^2,
\end{equation}
where $c=c(n,p)>0$.

Returning  to the proof of (\ref{cacci1}), we only need to estimate the last integral in (\ref{cacci2}). Applying H\"older's inequality, we have 
\begin{equation*}
\begin{aligned}
&\int_\Omega \eta^2 (\delta+ |\nabla_\e u^\e|^2)^{\frac{p-2+\beta}{2}} |Zu^\e|^2\\
\le & \Big(\int_\Omega \eta^{\beta+2}(\delta+|\nabla_\e u^\e|^2)^{\frac{p-2}{2}} |Z^\e u|^{\beta+2}\Big)^{\frac{2}{\beta+2}} 
\Big( \int_{spt(\eta)} (\delta+|\nabla_ \e u^\e|^2)^{\frac{p+\beta}{2}}\Big)^{\frac{\beta}{\beta+2}}.
\end{aligned}
\end{equation*}
Note that $|Zu^\e|\le 2 |\nabla_\e u^\e|$. The first integral in the right hand side of the above inequality can be estimated by (\ref{cacci4}), which can be absorbed to the left hand side of (\ref{cacci1}). 
This completes the proof of (\ref{cacci1}).

The following estimate for $Zu^\e$ follows from (\ref{cacci4}) and Theorem \ref{thm:Cacci1} is crucial  for the proof of the H\"older continuity of the horizontal gradient
of $u$ in Section 4. 

\begin{cor}\label{cor:Zu:inter}
Let $1<p<\infty$ and $u^\e$ be a weak solution of equation (\ref{maineqdeltaep}) with $\delta\ge0$.  Then for any $\beta\ge 2$ and every non-negative $\eta\in C^\infty_0(\Omega)$, we have that
\[
\int_\Omega \eta^{\beta+2} (\delta+ |\nabla_\e u^\e|^2)^{\frac{p-2}{2}}  |Zu^\e|^{\beta +2} \le c K^{\frac{\beta+2}{2}}\int_{spt(\eta)} (\delta+ |\nabla_\e u^\e|^2)^{\frac{p+\beta}{2}},
\]
where $K=||\nabla_\e\eta||^2_{L^\infty}+||\eta Z\eta||_{L^\infty}$ and $c=c(n,p,\beta)>0$.
\end{cor}





\section{$C^{1,\alpha}$-regularity of weak solutions of equation \eqref{maineq}} 
In this section, we discuss the H\"older continuity of horizontal gradient of solutions of equation (\ref{maineqdelta}) for $1<p<\infty$. Actually, as shown in \cite{Zhong, MZ}, this result holds for weak solutions to equation (\ref{maineq}). We refer to  \cite{CCLDO, Muk2} for more general equations, and also to \cite{CM} for equations involving H\"ormander vector fields of step two. The following theorem states that  weak solutions of equation
(\ref{maineqdelta}) are of class $C^{1,\alpha}$. 

\begin{thrm}\label{thm:Holder} Let $1<p<\infty$ and $u\in HW^{1,p}(\Omega)$ be a weak solution of equation (\ref{maineqdelta}) with $\delta\ge0$. Then 
the horizontal gradient $\nabla_0u$ is H\"older continuous. Moreover, there is a positive exponent $\alpha\le 1$, depending only on $n$ and $p$, such that
for any ball $B_{r_0}\subset \Omega$ and any $0<r\le r_0/2$, we have that
\begin{equation}\label{nablaHolder}
\max_{1\le l\le 2n} {\operatorname{osc}}_{B_r} X_l u\le  c\Big(\frac{r}{r_0}\Big)^\alpha\Big(\frac{1}{r_0^Q}\int_{B_{r_0}} (\delta+|\nabla_0 u|^2)^{\frac p 2}\Big)^{1/p},
\end{equation}
where $c=c(n,p)>0$.
\end{thrm}

There are two existing approaches for the proof of Theorem \ref{thm:Holder}. One is similar to that of DiBenedetto \cite{Di} and another to that of Tolksdorff \cite{To} and Lieberman \cite{Lieb}. Both approaches are based on De Giorgi's method \cite{De}. In this remarkable paper De Giorgi proved the local boundedness and H\"older continuity for functions satisfying certain
integral inequalities, nowadays known as De Giorgi's class of functions. In the remaining of this section, we discuss these two approaches.

First, the proof of Theorem \ref{thm:Holder} for the range $p>2$ in \cite{Zhong} is in the same line as that in \cite{Di}. 
Let $u\in HW^{1,p}(\Omega)$ be a solution of
equation (\ref{maineqdelta})  with $\delta>0$. Fix a ball $B_{r}\subset \Omega$. 
 We denote 
\[ \nu(r)=\max_{1\le l\le 2n}\sup_{B_{r}} \vert X_l u\vert.\]
and for $k\in \mathbb{R}$
\[
A^+_{k,r}=\{ x\in B_r: (u(x)-k)^+>0\}.
\]

Theorem \ref{thm:Holder} follows, as in \cite{Di} with a minor modification, from  a Caccioppoli inequality for $X_l u,
l=1,2,\ldots, 2n$,  see Lemma 4.3 of \cite{Zhong}. This Caccioppoli inequality shows that $X_lu$ belongs to a generalized version of De Giorgi's class. It states as follows.
For any $q\ge 4$, there exists $c=c(n,p,L,q)>0$ such that the
following inequality holds for any $1\le l\le 2n$, for any $k\in {\mathbb R}$ and for any
$0<r^\prime<r\le r_0/2$
\begin{equation}\label{cacci:Xu:k}
\begin{aligned}
&\int_{A^+_{k,r^\prime}} (\delta+ |\nabla_0 u|^2)^{\frac{p-2}{2}} \vert \nabla_0 X_l u\vert^2\, \\
\le&
\frac{c}{(r-r^\prime)^2}\int_{A^+_{k,r}}(\delta+ |\nabla_0u|^2)^{\frac{p-2}{2}}\vert(X_lu-k)^+\vert^2\,
 +c K\vert A^+_{k,r}\vert^{1-\frac{2} {q}}
\end{aligned}
\end{equation}
where $K =r_0^{-2}\vert
B_{r_0}\vert^{\frac 2 q}(\delta+\nu(r_0)^2)^{\frac{p} {2}}$.

The proof of (\ref{cacci:Xu:k}) 
is based on the estimate for $Zu$ in Corollary \ref{cor:Zu:inter}. It also involves an iteration argument.
We remark here that (\ref{cacci:Xu:k}) holds for $p>2$. In \cite{Di}, there is a version of (\ref{cacci:Xu:k}) for the case $1<p<2$
in the setting of Euclidean spaces. Unfortunately, we do not know if we have the analog of that in the setting of Heisenberg group.

Second, the proof of Theorem \ref{thm:Holder} in \cite{MZ} works for all $1<p<\infty$, and is similar to that  of Tolksdorff \cite{To} and  Lieberman  \cite{Lieb}
in the setting of Euclidean spaces. Following \cite{To, Lieb}, we consider the double 
truncation of the horizontal derivative $X_lu, l=1,2,..., 2n$, of the weak solution $u$ to equation (\ref{maineqdelta}) with $\delta>0$
\[ v=\min\big( \nu(r)/8, \max(\nu(r)/4-X_lu,0)\big). \]
The following Caccioppoli type inequality for $v$ was proved in \cite{MZ}. 
Let $\gamma>1$ be a number. We have the  Caccioppoli type inequality
 \begin{equation}\label{cacci:Xu:v}
  \int_{B_r} \eta^{\beta+4}v^{\beta+2}|\nabla_0 v|^2\,  \leq c(\beta+2)^2
\frac{| B_r|^{1-1/\gamma}}{r^2}
  \nu(r)^4\Big(\int_{B_r}\eta^{\gamma\beta}v^{\gamma\beta}\, \Big)^{1/\gamma}
\end{equation}
for all $\beta\ge 0$, where $c=c(n,p,\gamma)>0$. Once we obtain (\ref{cacci:Xu:v}), we may follow the same line as in \cite{Lieb} to prove Theorem 
\ref{thm:Holder}.

The proof of (\ref{cacci:Xu:v}) is also based on 
the integrability estimate  for $Zu$ in Corollary \ref{cor:Zu:inter}.
We consider the equation for $X_lu$, and use the usual testing function
\[\varphi=\eta^{\beta+4}v^{\beta+3}\] 
 where $\beta\ge 0$. 
The proof of (\ref{cacci:Xu:v}) in the case $p>2$ is easy, while that in the case $1<p<2$ is more involved.

\section{The  parabolic case}

In this section we sketch the proof of the regularity of weak solutions in the parabolic setting. As stated in the introduction, the development of the theory in the parabolic setting is not as advanced as in the degenerate elliptic setting, and at the moment one has only Lipischitz regularity in the range $2\le p \le 4$.
We state the main estimate again for the reader's convenience.

\begin{thrm}\label{main1-lip}  
Let $A_i$ satisfy the structure conditions \eqref{structure-zero} and 
let $u\in L^p((0,T), W_{\loc}^{1,p}(\Om))$ be a weak solution of \eqref{maineq-zero}  in $Q=\Om\times (0,T)$.  If $\ 2\le p \le 4$ then $|\nabla_0 u| \in L^\infty_{\loc}(Q)$ and $\partial_t u, Zu \in L^q_{\loc}(Q)$ for every $1\le q<\infty$. Moreover, in the range $2<p\le 4$ one has
that  for any $Q_{\mu, 2r}\subset Q$,
\begin{equation}\label{desired}
\sup_{Q_{\mu, r}}  |\nabla_0u|
\le C  \max\Big( \Big(\frac{1}{\mu r^{n+2}}\int\int_{Q_{\mu,2r}} (\delta+|\nabla_0u|^2)^{\frac{p}{2}}\Big)^{\frac{1}{2}}, \mu^{\frac{p}{2(2-p)}}\Big),
\end{equation}
where $C=C(n,p,\lambda, \Lambda, r,\mu)>0$ and $Q_{\mu,2r}$ is as in Theorem \ref{main-th-parabolic}. The special case $p=2$ has been studied before using linear techniques, and the estimates we obtain in that case are similar to \eqref{desired}, with the anisotropic cylinders being replaced with the classic parabolic cylinders. In the special case where there is no direct dependence on the space variable, i.e.  $A_i(x,\xi)=A_i(\xi)$, the parameters dependence  is more explicit, with
$C=C(n,p,\lambda, \Lambda) \mu^{\frac 1 2}>0$.
\end{thrm}

In order to prove this regularity result in the parabolic setting, 
we tried to mimic the procedure in the elliptic setting. 
As in the previous setting the horizontal derivatives satisfy a Cacciopoli inequality involving the derivative in the direction of the commutator $Zu$

\begin{lemma}[Lemma 3.5 \cite{CCG}] \label{lemma3.4} Let $u$ be a weak solution of \eqref{maineq} in $Q$, with $\delta>0$. For every  $\beta \geq 0 $ and non-negative $\eta\in C^{1}([0,T], C^\infty_0 (\Om))$ vanishing on the parabolic boundary of $Q$,
we have
\begin{equation*}
\begin{aligned}
&\frac{1}{\beta+2}\sup_{t_1<t<t_2}\int_\Omega (\delta+ |\nabla_0u|^2)^{\frac{\beta+2}{2}}\eta^2 +\int_{t_1} ^{t_2} \int_\Omega
(\delta + |\nabla_0u|^2)^{(p-2+\beta)/2} |\nabla_0^2u|^2 \eta^2\\
\leq & \, C
\int_{t_1} ^{t_2} \int_\Omega
(\delta + |\nabla_0u|^2)^{\frac{p+\beta}{2}} ( |\nabla_0\eta|^2 + |Z\eta| \eta)
+ 
\frac{C}{\beta+2}
\int_{t_1} ^{t_2} \int_\Omega 
(\delta + |\nabla_0u|^2)^{\frac{\beta+2}{2}}|\partial_t \eta| \eta \\
& +
C(\beta + 1)^4
\int_{t_1} ^{t_2} \int_\Omega
(
\delta + |\nabla_0u|^2)^{\frac{p-2+\beta}{
2}} |Zu|^2 \eta^2,
\end{aligned}
\end{equation*}
where $C = C(n, p, \lambda, \Lambda) > 0$, independent of $\delta$.
\end{lemma}

As before the main scope is to handle the $Zu$ term, in order to obtain a Cacciopoli type inequality for the horizontal derivatives which is independent of $Zu$. More precisely, we aim  
to prove the following Caccioppoli type inequality which is a parabolic analogue of Theorem \ref{thm:Cacci1}. 

\begin{thrm}\label{CCGMt2.8} 
Let $u^\e$ be a  solution of \eqref{approx1} in $\Om\times (0,T)$ and 
$B_{\e}(x_0,r) \times (t_0-r^2, t_0)$ a parabolic cylinder. 
Let  $\eta\in C^{\infty}( B_{\e}(x_0,r) \times (t_0-r^2, t_0])$
be a non-negative test function $\eta\leq 1,$ which vanishes on the parabolic boundary
such that  there exists a constant $C_{\lambda, \Lambda}>1$ for which
$||\partial_t \eta||_{L^\infty}\leq C_{\lambda, \Lambda} (1 + ||\nabla_\e\eta||^2_{L^\infty}).$
 Set $t_1 = t_0-r^2, t_2=t_0$. There exists a constant $C>0$ depending on $\ddelta$ 
$ p,$ and $\Lambda$ 
such that 
for all $\beta\ge 2$  one has
{\allowdisplaybreaks
\begin{align*}
& \int_{t_1} ^{t_2} \int_\Omega
\eta^2(\ddelta + |\nabla_\e u^\e|^2)^{(p-2+\beta)/2} \sum_{i,j=1}^{2n+1} |X^\e_i X^\e_j u^\e|^2 
 + \frac{1}{\beta+2} \max_{t \in (t_0-r^2, t_0]}\int_\Omega  (\ddelta+ |\nabla_\e u^\e|^2)^{\frac{\beta}{2}+1}\eta^2
\\
& \leq C (\beta + 1)^5 (||\nabla_\e\eta||^2_{
L^{\infty}} + ||\eta Z\eta||_{L^{\infty}} + 1)
\int_{t_1} ^{t_2} \int_{spt(\eta)}
(\ddelta + |\nabla_\e u^\e|^2)^{(p+\beta)/2}.
\end{align*}}
\end{thrm}

\subsection{The regularized $p-$Laplacian case} 

We first consider a simplified setting, where $\delta$ is fixed. 
The analogous of Lemma  \ref{lemma:rev} is the following

\begin{lemma} \label{cacciopoXZ}Set $T >t_2 > t_1> 0$. Let $u^\e$ be a weak solution of \eqref{approx1} in $Q=\Om\times (0,T)$. Let 
$\beta\ge 2$ and let $\eta\in C^1((0,T), C^{\infty}_0(\Om))$,  with $0\le \eta \leq 1$. For all $\alpha \leq 1$ there exist  constants
$C_\Lambda$, $C_\alpha = C(\alpha, \lambda, \Lambda)>0$ such that  
{\allowdisplaybreaks
\begin{align}\label{longwaydown}
& \int_{t_1}^{t_2} \int_\Om \eta^{\beta+2}(\ddelta+|\nabla_\e u^\e|^2)^{\frac{p-2}{2}} |Zu^\e|^\beta \sum_{i,j=1}^{2n+1} |X^\e_iX^\e_j u^\e|^2 
+ \int_\Om
 \eta^{\beta+2}|Zu^\e|^\beta |\nabla_\e u^\e|^2\bigg|_{t_1}^{t_2}
\notag\\
& \leq C_\al (\beta+1)^2(1+ |\nabla_\e \eta||_{L^\infty}^2)\int_{t_1}^{t_2} \int_\Om (\eta^{\beta} +  \eta^{\beta+4}) (\ddelta+|\nabla_\e u^\e|^2)^{\frac{p}{2}} |Zu^\e|^{\beta-2} \sum_{i,j=1}^{2n+1} |X^\e_iX^\e_j u^\e|^2
\notag\\
&+ C_{\Lambda}(\beta+1)^2
\int_{t_1}^{t_2}  \int_\Om
(\ddelta+|\nabla_\e u^\e|^2)^{\frac{p+2}{2}} |Zu^\e|^{\beta-2}    \eta^{\beta+2} 
\notag\\
& + \frac{2\al}{(1 + ||\nabla_\e \eta||^2)(\beta + 2)}\int_{t_1}^{t_2}\int_\Omega |Zu^\e|^{\beta+2} 
\eta^{\beta+3} |\partial_t \eta|   +
\int_{t_1}^{t_2} \int_{\Om} |Zu^\e|^{\beta}  |\nabla_\e u^\e|^2  \p_t ( \eta^{\beta+2}).
\notag
\end{align}
}
\end{lemma}

The last two terms come from the parabolic term, and they 
introduce a lack of homogeneity of the equation. Since here $\delta$ is fixed, we 
introduced a  simple trick to make these last two terms homogeneous. 
Recalling that $Z$ is obtained as a commutator of the horizontal vector fields and that $\eta\le 1$, we estimate
\begin{equation}\label{need delta>0 }
\int_{t_1}^{t_2}\int_\Omega |Zu^\e|^{2} \eta^3    \le  C_\delta\int_{t_1}^{t_2}\int_\Om
(\ddelta+|\nabla_\e u^\e|^2)^{\frac{p-2 }2} \sum_{i,j=1}^{2n+1} |X^\e_i X_j^\e u^\e|^2 \eta^2 .
\end{equation}

With this simple remark, we are now able to follow the same steps as in the subelliptic setting, and obtain Theorem \ref{main1-lip}. From here, since the equation is uniformly sub-parabolic, we finally establish with standard instruments the $C^{\infty}$ regularity of the solution for $\delta$ fixed (see \cite{CCG} for more details).

\subsection{The $2\le p \le 4$ case}

The $L^\infty$ estimate of gradient in case $2\leq p \leq 4$ the parabolic setting is similar to the stationary one. 
Again the main challenge is to deal the lack of homogeneity of the equation.

As in the stationary case, we start with an estimate of the derivative of the solution in the direction $Z$. 
Using this relation, a parabolic version of the Caccioppoli type estimate stated in  Proposition \ref{stimaZu} is established, which together with the 
 inequality \ref{SobXU} leads to an estimate of 
$L^{p+\beta}-$norm of $Zu^\e$

%

\begin{lemma}\label{mainlemmap}
Let $u^\e$ be a solution of \eqref{maineq} in $Q$, with $\delta>0$ and $2\leq p \leq 4$. 
Then for every $ \beta \geq 0$ and non-negative $\eta\in C^{1}([0,T], C^\infty_0 (\Om))$ vanishing on the parabolic boundary of $Q$, we have
\begin{equation}\label{mainestimate}
\begin{aligned}
\Big(\int_{t_1} ^{t_2}\int_\Omega &
|Zu^\e|^{p+\beta} \eta^{p+\beta} \Big)^{\frac{1}{p+\beta}}
\le  C(p+\beta) ||\nabla_\e\eta||_{L^\infty}\Big(\int\int_{spt(\eta)} (\delta+ |\nabla_\e u^\e|^2)^{\frac{p+\beta}{2}}\Big)^{\frac{1}{p+\beta}}\\
&  + 
C(p+\beta)||\eta\partial_t\eta||_{L^\infty}^{\frac{1}{2}}|spt(\eta)|^{\frac{p-2}{2(p+\beta)}}\Big(\int\int_{spt(\eta)}  (\delta+ |\nabla_\e u^\e|^2)^{\frac{p+\beta}{2}} \Big)^{\frac{4-p}{2(p+\beta)}}
\end{aligned}
\end{equation}
\end{lemma}

Having established an estimate for the derivative with respect to $Z$, we can proceed to the estimate of the intrinsic gradient $(X_1u, \ldots, X_mu)$. Differentiating the equation with respect to $X^\e_l$ we obtain the differential equation satisfied by $X^\e_lu$, and from this, a 
Cacciopoli type inequality, analogous to the Theorem \ref{thm:Cacci1} established in the stationary setting. More in detail,
%
%

\begin{prop}\label{cor1p}
Let $u^\e$ be a weak solution of \eqref{maineq} in $Q$, with $\delta>0$ and $2\leq p \leq 4$. 
Then for every $ \beta \geq 0$ and non-negative $\eta\in C^{1}([0,T], C^\infty_0 (\Om))$ vanishing on the parabolic boundary of $Q$, we have
\begin{equation}\label{maincacci}
\begin{aligned}
&\sup_{t_1<t<t_2}\int_\Omega (\delta+ |\nabla_\e u^\e|^2)^{\frac{\beta+2}{2}}\eta^2 +\int_{t_1} ^{t_2} \int_\Omega
(\delta + |\nabla_\e u^\e|^2)^{(p-2+\beta)/2} |\nabla_\e^2u^\e|^2 \eta^2\\
\leq & C(p+\beta)^7\big( ||\nabla_\e\eta||_{L^\infty}^2+||\eta Z\eta||_{L^\infty}\big)\int\int_{spt(\eta)} (\delta+ |\nabla_\e u^\e|^2)^{\frac{p+\beta}{2}}\\
& + 
C(p+\beta)^7||\eta\partial_t\eta||_{L^\infty}|spt(\eta)|^{\frac{p-2}{p+\beta}}\Big(\int\int_{spt(\eta)}  (\delta+ |\nabla_\e u^\e|^2)^{\frac{p+\beta}{2}} \Big)^{\frac{\beta+2}{p+\beta}},
\end{aligned}
\end{equation}
where $C=C(n,p,\lambda,\Lambda)>0$.
\end{prop}

From the latter, the Lipschitz bound in Theorem \ref{main1-lip} follows through a modification of Moser's iteration method.

\section{Open problems}

The study on the subelliptic equation (\ref{maineq}) and its parabolic counterpart (\ref{maineq-zero}) is still in the early stage. In the following, we list some open problems, concerning the regularity and property of solutions to these equations.

(i) {\it Regularity for equation (\ref{maineq}) in a Carnot group}.
The $C^{1,\alpha}$-regularity for equation (\ref{maineq}) can be easily extended to any step two Carnot group. Actually,  the papers \cite{CCLDO} and \cite{CM} show extensions beyond the group setting, but within the step two hypothesis. The $C^{1,\alpha}$-regularity for equation (\ref{maineq}) in a general Carnot group is completely open. Even the Lipschitz regularity is not known in this case. We refer to \cite{CCLDO} for the connect between the regularity of equation (\ref{maineq}) and 1-quasiregularity in the case $p=2n+2$. We do believe that the Cacciopoli inequality in Theorem \ref{thm:Cacci1} still holds in this setting. But it seems that the approach in \cite{Zhong} does not work and it requires new ideas to prove such a Cacciopoli inequality.

(ii) {\it Regularity for $p$-Laplacian system in the Heisenberg group}. We consider the $p$-Laplacian system 
\begin{equation}\label{plaps} 
\sum_{i=1}^{2n} X_i  \Bigg( |\nabla_0 u|^{p-2} X_i u \Bigg)=0, 
\end{equation}
where $u$ is a vector-valued function defined a domain of $\Hn$.
The $C^{1,\alpha}$-regularity for this equation is not known. In the Euclidean setting, it is known that solutions are in the class  $C^{1,\alpha}$ \cite{U,Tolksdorff}.

(iii) {\it Singular $p$-harmonic functions in the Heisenberg group}. Let $\Omega$ be an open subset of $\mathbb{R}^n$ such that $0\in \Omega$. Let $u$ be a $p$-harmonic function in $\Omega\setminus \{0\}$ for $1<p\le n$. Serrin \cite{Se, Se2} showed that if $u$ is bounded from below and not bounded from above in  $\Omega\setminus\{0\}$, then there is a constant $c\ge 1$ such that
\[
\frac{1}{c} u_0\le u \le c u_0
\]
in a neighbourhood of the origin, where $u_0$ is, up to a constant, the fundamental solution to the $p$-Laplacian equation (\ref{plap}) 
\begin{equation*}
u_0(x)=\begin{cases} |x|^{(p-n)/(p-1)} \quad &\text{ if } 1<p<n;\\
\log (1/|x|), &\text{ if } p=n.
\end{cases}
\end{equation*}
In \cite{KV}, Kichenassamy and Veron improved Serrin's result and showed that there is a constant $\gamma$
such that
\[
u-\gamma u_0\in L^\infty_\loc (\Omega),
\]
and that
\[
\lim_{x\to 0}|x|^{(n-1)/(p-1)}(\nabla u(x)-\gamma \nabla u_0(x))=0.
\]
Naturally we ask if similar results hold for
the singular $p$-harmonic functions in the Heisenberg group. The fundamental solution to equation (\ref{plap})  has an explicit form in the Heisenberg group and in H-type groups, see Theorem 2.1 of \cite{CDG}. A new difficulty in this setting is that the horizontal gradient of such fundamental solution vanishes on the center of the group. This prevents one from using the existing strong comparison theorems for $p$-harmonic functions. We remind the readers that the strong comparison theorem for $p$-harmonic function is still open even in the Euclidean spaces $\mathbb R^n$, for $n\ge 3$ (the planar case was established by Manfredi in \cite{Manfredi-comp}).

(iv) {\it Regularity for equation (\ref{maineq-zero}) in the Heisenberg group}.
For the non-stationary case, we know that the weak solutions to equation (\ref{maineq-zero}) in the Heisenberg group, or in any step two Carnot group, are Lipschitz continuous in the case $2\le p\le 4$. This is the only existing regularity result.
The $C^{1,\alpha}$-regularity is open even in the case $2<p\le 4$. We naturally ask if we have similar regularity theory for equation (\ref{maineq-zero}) for the full range $1<p<\infty$ in the Heisenberg group, or more generally in Carnot groups, as that in the setting of Euclidean spaces. Our understanding on equation (\ref{maineq-zero}) is limited. We need new ideas to handle this equation. For example, to prove the Lipschitz continuity of the solutions to equation (\ref{maineq-zero}), we do believe that Theorem \ref{CCGMt2.8} holds for all $1<p<\infty$. 
But, if we would mimic the proof in the stationary case, we would end up with a term, coming from $\p_t u$, for which the existing techniques do not seem to apply directly.

(v) {\it Homogenous parabolic $p-$Laplacian}. It would be interesting to study the equation
$$\partial_t u = \sum_{i=1}^{2n} X_i^2 u + (p-2) \frac{ X_i u X_j u}{|\nabla_0 u|^2}X_iX_j u,$$
and try to recover the results in \cite{JSi}. The study of this type of homogenous $p-$Laplacian seems to be completely open in the subriemannian setting.

\end{document}